\documentclass{article}
\usepackage[square,numbers]{natbib}
\usepackage{times}
\usepackage{amsfonts}
\usepackage{amsmath}
\usepackage{amsthm}
\usepackage{amssymb}
\usepackage{tikz}
\usepackage{pgfplots}
\usepackage{comment}
\usetikzlibrary{positioning}
\usetikzlibrary{arrows,decorations.markings}
\usetikzlibrary{shapes.geometric, arrows}
\usepackage[font=small,labelfont=bf]{caption}
\usetikzlibrary{intersections, calc, angles}
\usepackage{subcaption}
\usetikzlibrary{decorations.pathmorphing}
\usetikzlibrary{decorations.pathreplacing,calligraphy}
\usepackage{hyperref}
\usepackage{relsize}
\newtheorem{theorem}{Theorem}
\newtheorem{example}{Example}
 \newtheorem{corollary}{Corollary}
\newtheorem{definition}{Definition}

\newtheorem{lemma}{Lemma}

\tikzset{
bicolor/.style 2 args={
  thick,dashed,dash pattern=on 3pt off 3pt,#1,
  postaction={draw,dashed,thick,dash pattern=on 3pt off 3pt,#2,dash phase=3pt}
  },
swiggle/.style={
    -stealth,decorate,decoration={snake,amplitude=3pt,pre length=2pt,post length=3pt}
  }
}
\usepackage{fancyhdr}
\pgfplotsset{compat=1.18}

\begin{document}
\sloppy

\title{Kharitonov's Theorem with Degree Drop: a Wronskian Approach}
\author{Anthony Stefan\footnote{Email: astefan2015@my.fit.edu}\; and Aaron Welters\footnote{Email: awelters@fit.edu}\\Florida Institute of Technology\\Department of Mathematics and Systems Engineering\\Melbourne, FL USA\\
\\
Jason Elsinger\footnote{Email: jason\_elsinger@dpsnc.net}\\Charles E.\ Jordan High School\\Department of Mathematics\\Durham, NC USA}

\date{}
\maketitle

\begin{abstract}
In this paper, we present a simplified proof of Kharitonov's Theorem, an important result on determining the Hurwitz stability of interval polynomials. Our new approach to the proof, which is based on the Wronskian of a pair of polynomials, is not only more elementary in comparison to known methods, but is able to handle the degree drop case with ease.
\end{abstract}

\section{Introduction}
Recently, the theory of stable polynomials has seen a surge in interest as it has played a key role in the solving of several long-standing open problems \cite{09BBa, 09BBb, 09BBc, 10BB, 11DW, 12RP, 15PB, 15MSS, 15EM, 16GKVW, 19GN, 22CS, 20KS, 24ACV}. Part of the success is based on novel applications and multivariable generalizations of some classical results in the stability of polynomials, such as the Hermite-Biehler Theorem \cite[Theorem 6.3.4]{02RS} and the Hermite-Kakeya-Obreschkoff Theorem \cite[Theorem 6.3.8]{02RS}, where the Wronskian of two polynomials 
played a key role (e.g., see \cite{09BBa}--\cite{10BB} and the survey \cite{11DW}). In addition, there has recently been an increased interest in generalizing stability results to other classes such as entire functions and matrices (e.g., see \cite{11VK, 14AD, 19OK, 21ZD}) as well as convex invertible cones of functions (e.g., see \cite{11AL, 13AL, 21HN, 21AL, 22HM, 24AL}). 

Motivated by this, we consider Kharitonov's Theorem \cite{78VK} (see also \cite{88PV, 05HP, 22AE}), a classical result in polynomial stability theory which has many applications in linear systems and control theory \cite{93JA, 95BCK, 05HP, 08AP, 17YT}. In this paper, we give a new elementary proof of this theorem (i.e., Theorem \ref{ThmKhar} in Sec.\ \ref{sec:KharThmStatement} and our proof in Sec.\ \ref{sec:PrfKharThm}) using the Wronskian as the main tool to do so. As a consequence, we have the shortest path to the proof in comparison to other methods (cf.\ \cite{78VK, 93JA, 94RB, 95BCK, 96DH, 99WT, 00KT, 05HP, 08AP}) and without sacrificing clarity in the exposition. 

Let us briefly elaborate on these points. First, in our new approach we use the result on the continuity of the roots of a polynomial (see Lemma \ref{LemMain}) in the proof of Lemma \ref{LemBddCrossnThmAndConvexCombStablePolysEndPts}, which is our extended version of the ``Boundary Crossing Theorem" (cf.\ \cite{93JA, 95BCK}) that allows for a ``degree drop" (see also Def.\ \ref{def:degreedropP} below). Second, our approach emphasizes the Wronskian of two polynomials [as defined by (\ref{DefWronkianOfhandg}) below]. Next, while our paper is inspired by the proof using the Bezoutian \cite{99WT} (see also \cite{93AF, 05OO}), 
the approach we take is simpler as it uses only the basic properties of the Wronskian (e.g., Lemma \ref{LemPosWronskian}). Another common approach to the proof of Kharitonov's Theorem involves a moving Kharitonov rectangle (see Fig.\ \ref{FigKharitonovRectangleUndirected} below) 
together with the ``zero-exclusion condition'' (e.g., see \cite[Remark 5.7.3, Lemma 5.7.9, and Sec.\ 5.8]{94RB}). However, this approach leaves the analyst with something to be desired in terms of the finer details of the argument. In addition, the zero-exclusion condition uses a weaker version of the continuity of roots of a polynomial (e.g., see \cite[Lemma 4.8.2]{94RB}) that can only deal with the degree invariant case and hence cannot treat the degree drop case, which we can do easily. Finally, as our proof is kept at a more elementary level and is self-contained (starting from Lemma \ref{LemMain} in Sec.\ \ref{sec:PrelimResults}), it will not only be of interest to mathematicians but will also be of independent interest in the systems theory, engineering, and control community, where Kharitonov's Theorem is well-known and where elementary proofs of classical results in stability theory are desired (for instance, see \cite{90CMB, 95GM}).

Let us give a motivation for Kharitonov's Theorem before we state  it. Consider any ordinary differential equation that is homogeneous, linear, and  scalar with constant real coefficients
\begin{gather*}
    \sum_{i=0}^na_i\frac{d^iy}{dt^i}=0,
\end{gather*}
where $a_i\in \mathbb{R}, i=0,\ldots, n$. These ODEs are called \textit{stable} \cite{00FGv2} if every classical solution $y=y(t)$ satisfies $\lim_{t\rightarrow+\infty}y(t)=0$. In particular, solutions of the form $y(t)=e^{\lambda t},$ where $ \lambda\in \mathbb{C}$, correspond to roots of the characteristic polynomial
\begin{gather*}
    p(z)=\sum_{i=0}^na_iz^i.
\end{gather*}
Thus, in order for the ODEs above to be stable, all the roots $\lambda$ of the polynomial $p=p(z)$ must have negative real part, i.e.,
\begin{gather*}
    p(\lambda)=0 \iff \operatorname{Re}\lambda <0.
\end{gather*}
Next, if there is uncertainty in the range of the parameters $a_i, i=0,\ldots, n$, then one is lead to consider stability of a whole family $P$ of real polynomials
\begin{gather}
    P=\left\{p:p=\sum_{i=0}^na_iz^i,a_i\in [a_i^{-},a_i^{+}] \right\},\label{DefIntervalPolyOrdern}
\end{gather}
where $a_i^{-},a_i^{+}$ are fixed real numbers satisfying $a_i^{-}\leq a_i^{+}, i=0,\ldots, n$, with either $a_n^{-}\not=0$ or $a_{n}^+\not=0$. Finally, as the family $P$ is infinite (assuming as we do that $a_i^-<a_i^+$ for at least one $i$), it becomes problematic to test whether all polynomials $p\in P$ have roots that lie in the open left-half plane, i.e., in the set $\{z\in\mathbb{C}:\operatorname{Re}z<0\}$ (whereas testing a small number of polynomials is not considered an issue\label{AFootnote}\footnote{The \textit{Routh-Hurwitz problem} \cite{77SS} is to find necessary and sufficient conditions for a given polynomial $p$ to have all its roots lying in the open left-half plane. Efficient algorithms do exist to test this such as the Routh-Hurwitz criterion, 
which can be proved using just continuity of the roots of the polynomials (see \cite{95GM}). This then gives another application of Lemma \ref{LemMain} in the stability theory of polynomials (also see \cite{90CMB}).}). 
It then becomes desirable to find a finite subset $V\subseteq P$ such that all the roots of the polynomials in $V$ lying in the open left-half plane implies all the roots of $P$ do as well. This is why Kharitonov’s Theorem (see Theorem \ref{ThmKhar}) is important, as it tells us that such a test exists with only the four element set $V=\{k_1,k_2,k_3,k_4\}$ of Kharitonov polynomials (see Def.\ \ref{DefKharPoly}).


\section{Kharitonov's Theorem: The  Statement}\label{sec:KharThmStatement}
In this section we state precisely Kharitonov's Theorem and give an example of its use.
\begin{definition}\label{DefHurwitzStablePolynomial}
    A polynomial $p(z)$ is called Hurwitz stable if all its roots lie in the open left half-plane $\{z\in\mathbb{C}:\operatorname{Re}z<0\}$.
\end{definition}

\begin{definition}\label{def:degreedropP}
    The set of real polynomials $P$ in (\ref{DefIntervalPolyOrdern}) is called an interval polynomial of order $n$ 
    and is degree invariant if $a_n^{-}a_{n}^+\not=0$; otherwise, it has degree drop. 
    Moreover, $P$ is said to be Hurwitz stable (or robustly stable) if every polynomial in $P$ is Hurwitz stable.
\end{definition}

\begin{definition}\label{DefKharPoly}
    The Kharitonov polynomials 
    associated with an interval polynomial $P$ in (\ref{DefIntervalPolyOrdern}) are the following four polynomials $k_j=k_j(z), j=1,2,3,4$, in $P$ defined as
\begin{gather}
    \label{DefKharPolys}\left\{\begin{array}{l}
        k_1=a_0^{-}+a_1^{-}z+a_2^{+}z^2+a_3^{+}z^3+\cdots;\\
    k_2=a_0^{+}+a_1^{-}z+a_2^{-}z^2+a_3^{+}z^3+\cdots;\\
    k_3=a_0^{+}+a_1^{+}z+a_2^{-}z^2+a_3^{-}z^3+\cdots;\\
    k_4=a_0^{-}+a_1^{+}z+a_2^{+}z^2+a_3^{-}z^3+\cdots.
    \end{array}\right.
\end{gather}
\end{definition}

The next result, called \textit{Kharitonov’s Theorem}, characterizes the Hurwitz stability of an interval polynomial $P$. It was first proven in 1978 by Vladimir Kharitonov 
\cite{78VK}\footnote{The problem of finding necessary and sufficient conditions for Hurwitz stability of a given interval polynomial $P$ was first posed in 1953 by S.\ Faedo \cite{53SF}. For more on this, see \cite{22AE}.} in the case $P$ is degree invariant. In 1992, T.\ Mori and H.\ Kokame \cite{92MK} showed how to extended his theorem to the case $P$ has degree drop, but they required the four Kharitonov polynomials along with two other distinguished polynomials in $P$. In 1996, Hernandez and Dormido \cite{96DH} showed that those two additional polynomials are not needed. After this, there were several proofs of the theorem given in the degree drop case \cite{99WT, 00KT} (see also \cite{06YF}), each with the goal of providing an elementary proof.
\begin{theorem}[Kharitonov’s Theorem]\!\!\!\footnote{There is also an extension of this theorem to polynomials with complex coefficients \cite{78VKb, 87BS, 08OT, 22AE}, but we do not treat it here in order to keep the presentation as simple as possible.}\label{ThmKhar} 
An interval polynomial $P$ is Hurwitz stable
if and only if the Kharitonov polynomials $k_j,j=1,2,3,4,$ associated with $P$ are Hurwitz stable.
\end{theorem}

The following example demonstrates a use of this theorem. 
\begin{figure}[!ht]
    \begin{subfigure}{0.48\textwidth}
        \includegraphics[scale=0.1475]{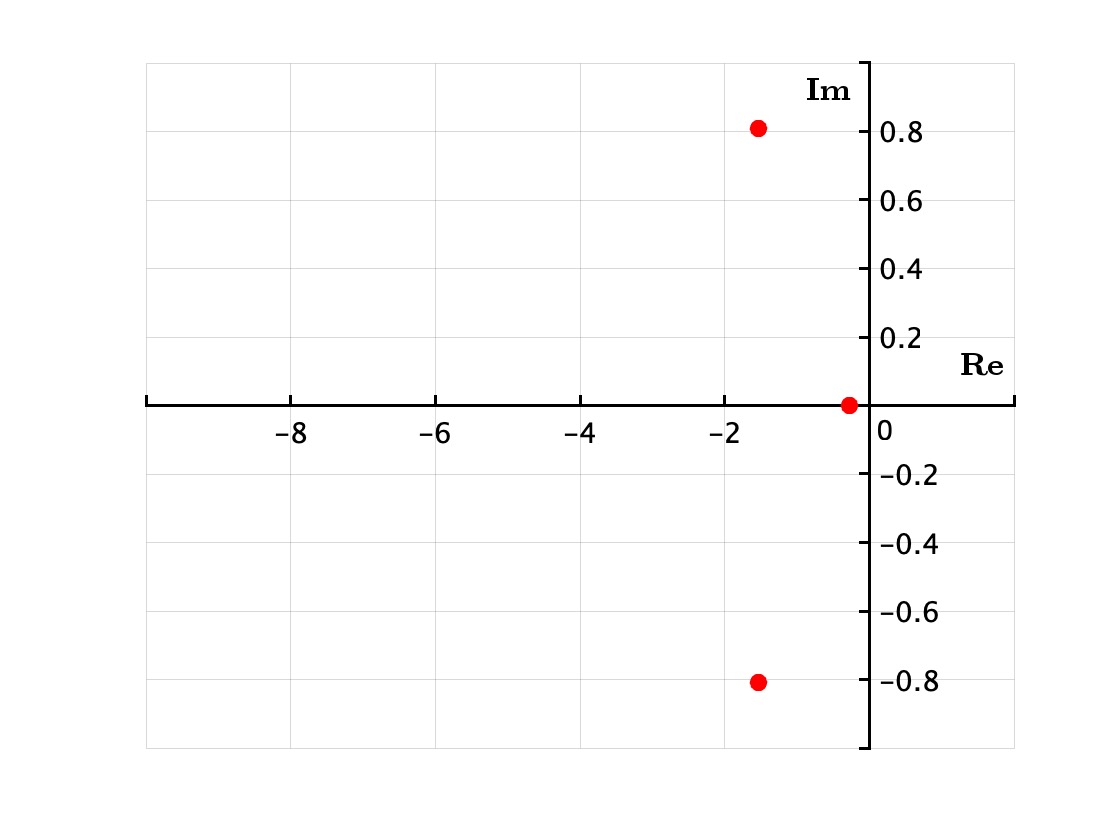}
    \end{subfigure}  
    \begin{subfigure}{0.48\textwidth}
        \includegraphics[scale=0.1475]{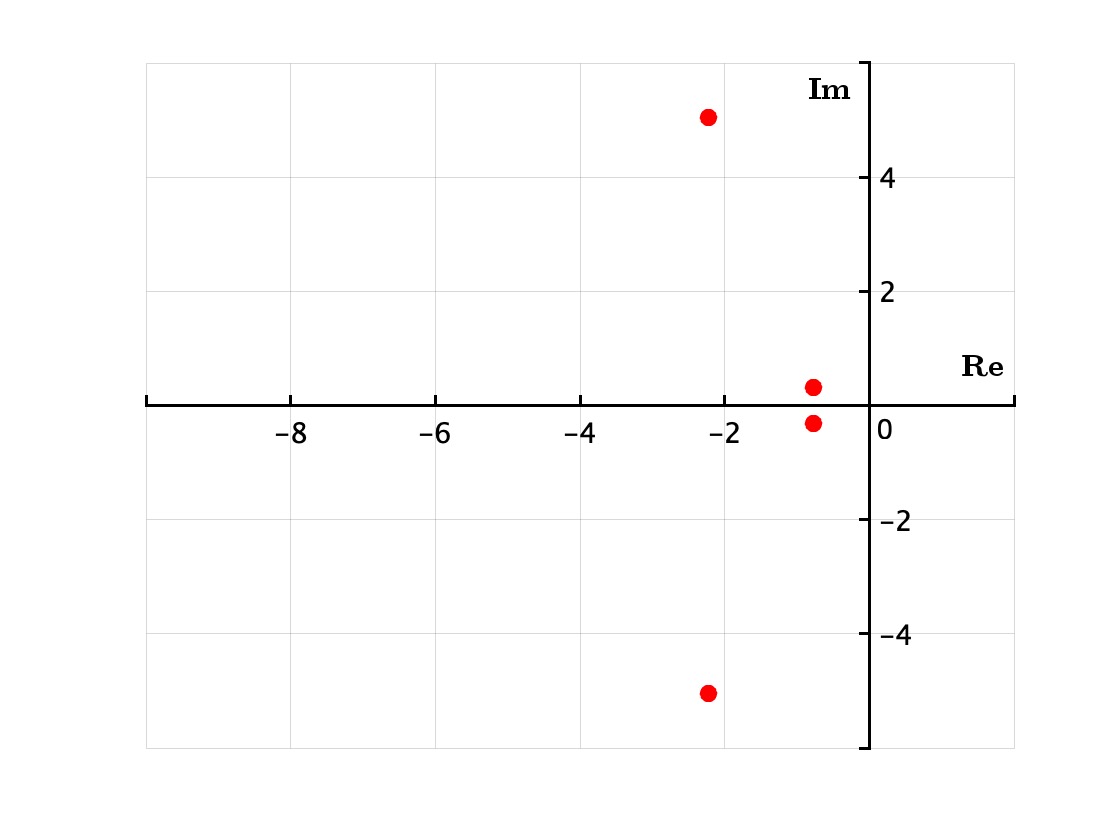}
    \end{subfigure} \\
    \begin{subfigure}{0.48\textwidth}
      \includegraphics[scale=0.1475]{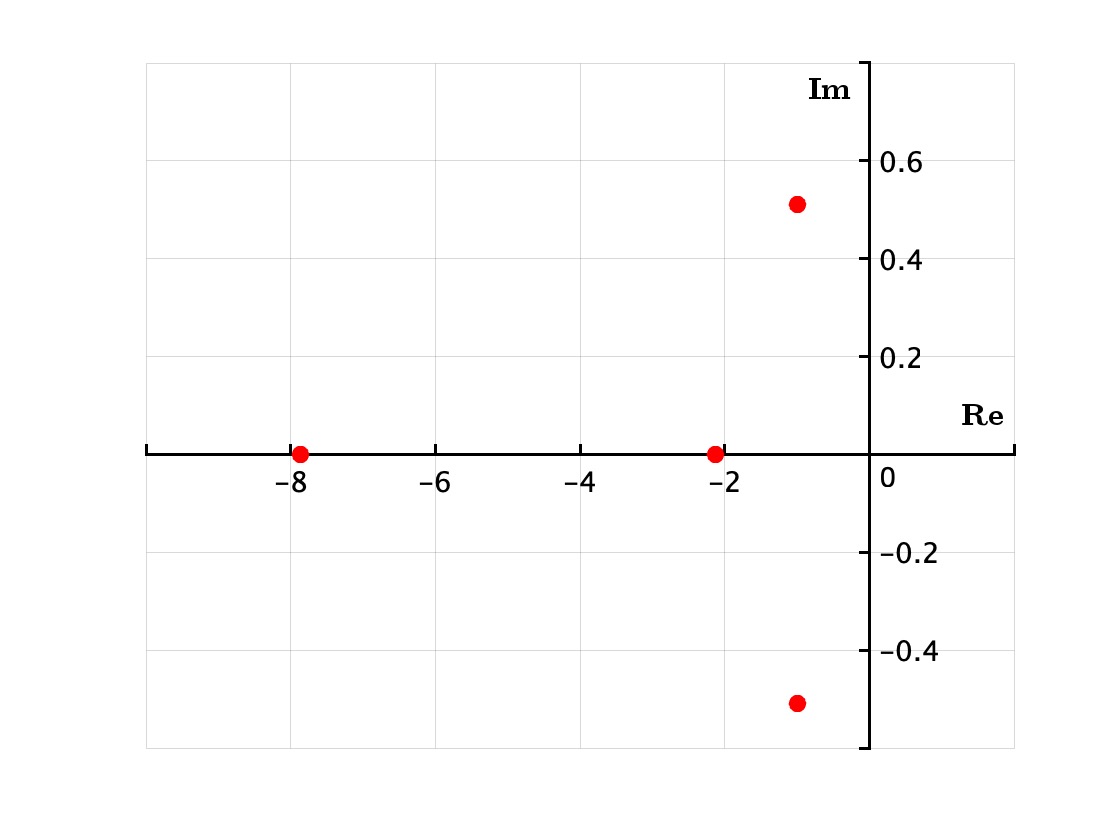}
    \end{subfigure}
    \begin{subfigure}{0.48\textwidth}
       \includegraphics[scale=0.1475]{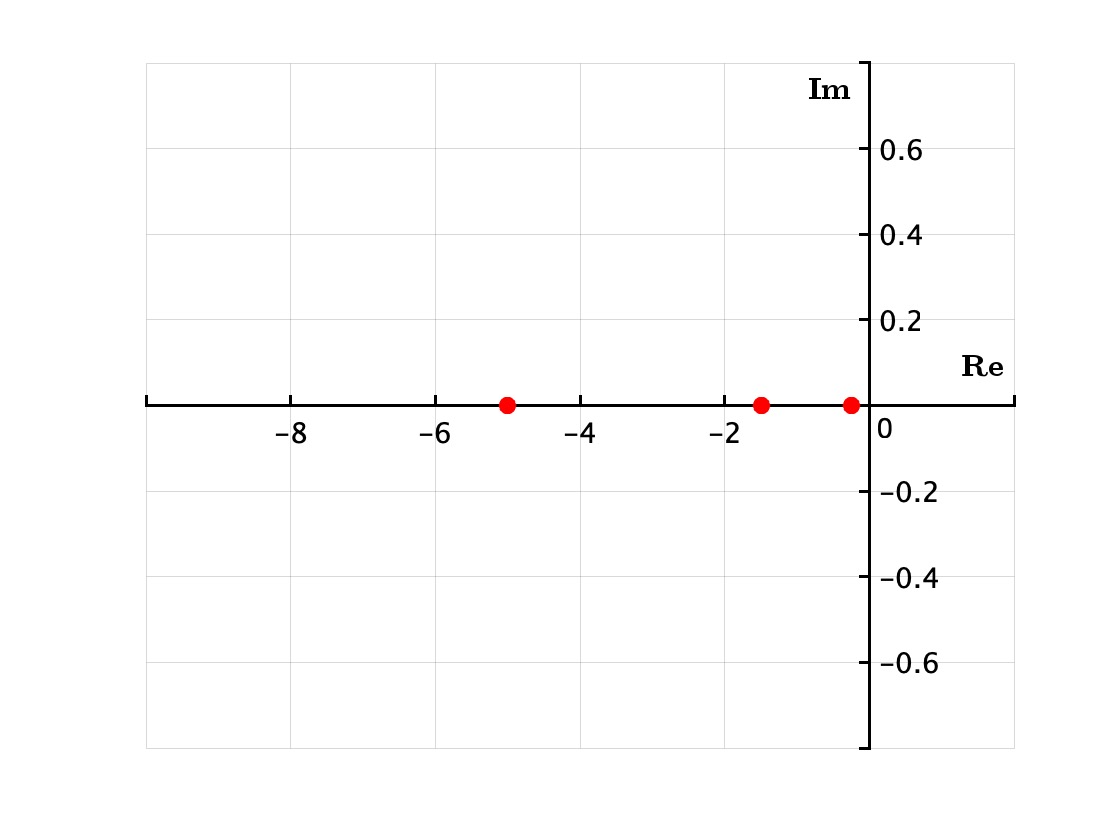}
    \end{subfigure}
    \caption{Plots of all the roots in the complex plane of the Kharitonov polynomials $k_j, j=1,2,3,4$, given by (\ref{DefKharPolysExample}) in Example \ref{example2} as generated by MATLAB. Here the roots are shown as the red dots for $k_1(z)$ on the top left, $k_2(z)$ on the bottom left, $k_3(z)$ on the top right, and $k_4(z)$ on the bottom right. Using the $\operatorname{roots}(\cdot)$ command in MATLAB, one obtains the approximation for the roots of $k_1(z)$ as $z=-0.28, -1.53\pm0.81i$, the roots of $k_2(z)$ as $z=-7.87, -2.13, -1.00\pm0.51i$, the roots of $k_3(z)$ as $z=-2.23\pm5.05i, -0.77\pm0.31i$, and the roots of $k_4(z)$ as $z=-5.10, -1.32, -0.25$. As the real parts of all these roots are negative (which can also be shown using the methods described in Footnote \ref{AFootnote}), Kharinotov's Theorem (i.e., Theorem \ref{ThmKhar}) implies that every polynomial $p$ in the interval polynomial $P$, as given by (\ref{ExInvervalPolyPartI}) and (\ref{ExInvervalPolyPartII}) in this example, is a Hurwitz stable polynomial.
    }
    \label{fig:Kpoly}
\end{figure}

\begin{example}\label{example2}
Let $P$ be the interval polynomial {\upshape{(}}of order $n=4$ with degree drop{\upshape{)}} whose elements are of the form
\begin{gather}
    p(z)=a_4z^4+a_3z^3+a_2z^2+a_1z+a_0\label{ExInvervalPolyPartI}
\end{gather}
with coefficients belonging to the following intervals: 
\begin{gather}
    a_0\in[10, 21], a_1\in[46, 50], a_2\in[38, 40], a_3\in[6, 12], a_4\in[0, 1].\label{ExInvervalPolyPartII}
\end{gather}
Then the Kharitonov polynomials are given by
\begin{gather}
    \label{DefKharPolysExample}\left\{\begin{array}{l}
        k_1=10+46z+40z^2+12z^3;\\
    k_2=21+46z+38z^2+12z^3+z^4;\\
    k_3=21+50z+38z^2+6z^3+z^4;\\
    k_4=10+50z+40z^2+6z^3.
    \end{array}\right.
\end{gather}
Each of these four polynomials are Hurwitz stable, which can be verified either by the Routh-Hurwitz criterion (see Footnote \ref{AFootnote}) or numerically (see Fig.\ \ref{fig:Kpoly}). Thus, by Kharitonov’s Theorem (i.e., Theorem \ref{ThmKhar}), the entire family $P$ is Hurwitz stable.
\end{example}

\section{Preliminary Results}\label{sec:PrelimResults}
In order to prove Kharitonov's Theorem, we will need some notation and preliminary results. First, consider any pair of polynomial $q(z)$, $p(z)$ in the form:
\begin{align}
    q(z)&= b_nz^n + b_{n-1}z^{n-1} + \cdots + b_1z+b_0, \label{Def_Poly_q}\\
    p(z)&= a_nz^n + a_{n-1}z^{n-1} + \cdots + a_1z+a_0, \label{Def_Poly_p}
\end{align}
for some $a_i, b_i\in \mathbb{C}$ and $i=0,\ldots, n$. As usual, the degree of a nonzero polynomial $p$ will be denoted by $\deg p$.

The next result on the continuity of the roots of polynomials in terms of their coefficients is well-known (see, for instance, \cite[Theorem 3]{89CC}).

\begin{lemma}\label{LemMain}
    Suppose $q(z)$, given by (\ref{Def_Poly_q}), is a nonzero polynomial and, if $\deg q\geq 1$, let $\zeta_1, \ldots, \zeta_d$ be all the distinct roots of $q$ with $m_j$ denoting the multiplicity of the root $\zeta_j$, for $j=1,\ldots, d$. Then for each sufficiently small $\varepsilon>0$, there exists a $\delta>0$ such that any polynomial $p(z)$ of the form (\ref{Def_Poly_p}) with $\max_{0\leq i\leq n}|a_i-b_i|<\delta$ will have exactly $m_j$ roots (counting multiplicities) whose distance from $\zeta_j$ is less than $\varepsilon$ for $j=1,\ldots,d$ and the remaining $\deg p - \deg q$ roots (counting multiplicities) have distance from $0$ greater than $\frac{1}{\varepsilon}$.
\end{lemma}
 
The following lemma, which is new, is an extended version of a well-known result, called the Boundary Crossing Theorem\footnote{A result that is attributed to Frazer and Duncan in their 1929 paper \cite{29FD} (cf. \cite{93JA}).} (see, for instance, \cite[Theorem 4.3]{93JA}, \cite[Theorem 1.4]{95BCK}), and is a simple corollary of Lemma \ref{LemMain}.

\begin{lemma}
\label{LemBddCrossnThmAndConvexCombStablePolysEndPts}
    Consider any family of polynomials
    \begin{gather*}
         p_{t}(z)=a_n(t)z^n\cdots +a_{1}(t)z+a_0(t),\hspace{0.2in} t\in [a,b],
     \end{gather*}
     where $a_j:[a,b]\rightarrow \mathbb{C}$ is a continuous function for each $j=0,\ldots, n$. If $p_a(z)$ is a Hurwitz stable polynomial such that $a_n(t)\not=0$ for every $t\in [a,b)$, then either $p_{t}(z)$ is Hurwitz stable for every $t\in[a,b]$ or there exists $(t_*,\omega_*)\in (a,b]\times \mathbb{R}$ such that $p_{t_*}(i\omega_*) =0$.
\end{lemma}
\begin{proof}
    Suppose there exists a $\tau\in [a,b]$ for which $p_{\tau}(z)$ is not Hurwitz stable. Then from the hypotheses that $p_a(z)$ is Hurwitz stable and  $a_n(t)\not=0$ for all $t\in[a,b)$, it follows from Lemma \ref{LemMain} that 
    \begin{gather*}
        t_*=\inf\{t\in [a,b]:p_t(z)\text{ is not Hurwitz stable}\}\in (a,\tau]. 
    \end{gather*}
     This implies that $p_t(z)$ is Hurwitz stable for all $t\in [a,t_*)$, and that there exists a sequence  $t_k\in [t_*,b]$, with $t_k\rightarrow t_*$ as $k\rightarrow\infty$, such that $p_{t_k}(z)$ is not a Hurwitz stable polynomial for all $k$. If $p_{t_*}(z)$ is the zero polynomial, then the lemma is true. Hence, assume that $p_{t_*}(z)$ is not the zero polynomial. Then by Lemma \ref{LemMain}, all the roots of $p_{t_*}(z)$ must lie in the closed left-half plane since $p_t(z)$ is Hurwitz stable for all $t\in [a,t_*)$. 

     Consider the case $t_*=b$. Then by our hypotheses we must have that $p_{b}(z)$ is not Hurwitz stable, but has all its roots in the closed left-half plane implying it must have a purely imaginary root, i.e., there exists a $\omega_*\in \mathbb{R}$ such that $p_b(i\omega_*)=0$. 
     
     Consider the case $t_*\not=b$. As each $p_{t_k}(z)$ is not a Hurwitz stable polynomial, then $p_{t_k}(z)$ has a root in the closed right-half plane for all $k$. Hence, as $p_{t_k}\rightarrow p_{t_*}$ as $k\rightarrow\infty$ and $a_n(t)\not=0$ for all $t\in [a,b)$, it follows by Lemma \ref{LemMain} that $p_{t_*}(z)$ must also have at least one root in the closed right-half plane. Since $p_{t_*}(z)$ also has all its roots in the closed left-half plane, then it must have at least one purely imaginary root, i.e., there exists a $\omega_*\in \mathbb{R}$ such that $p_{t_*}(i\omega_*)=0$. This completes the proof.
\end{proof}

Before move on, it is worth pointing out that the hypotheses of the above lemma, specifically, that the change in the degree of the polynomial $p_t(z)$ can occur only at the right endpoint $t=b$ (where it can only drop in degree), cannot be weakened as the next two examples show.
\begin{example}
    Consider the family of polynomials $p_t(z)=a_1(t)z+a_0(t), t\in [0,1]$, where $a_j:[0,1]\rightarrow \mathbb{C}, j=0,1$, are the continuous functions defined by 
    \begin{gather*}
        a_0(t)=1,\;a_1(t)=(2t-1)^2-1.
    \end{gather*}
    Then $p_0(z)=p_1(z)=1$ is a Hurwitz stable polynomial and $p_t(i\omega)\not=0$ for every $(t,\omega)\in [0,1]\times\mathbb{R}$, but $p_t(z)$ is not a Hurwitz stable polynomial for each $t\in (0,1)$ since its only root is $z=-1/a_1(t)\in [1,\infty)$.  
    Notice that Lemma \ref{LemBddCrossnThmAndConvexCombStablePolysEndPts} does not apply since $a_1(c)=0$ has two solutions: $c=0$ and $c=1$.
\end{example}
\begin{example}
    Consider the family of polynomials $p_t(z)=a_1(t)z+a_0(t), t\in [0,1/2]$, where $a_j:[0,1/2]\rightarrow \mathbb{C}, j=0,1$, are those in the previous example. Then $p_0(z)$ is a Hurwitz stable polynomial, $a_1(t)=0$ only if $t=0$, and $p_t(i\omega)\not=0$ for every $(t,\omega)\in [0,1/2]\times\mathbb{R}$, but $p_t(z)$ is not a Hurwitz stable polynomial for each $t\in (0,1/2]$.
\end{example}

The next lemma is a classical result in stability theory (see, for instance, \cite[Prop. 11.4.2]{02RS}, \cite[Theorem 9.1]{08AP}).
\begin{lemma}[Stodola's Rule]\label{LemStodolasRule}
    If $p(z)=a_nz^n+\cdots +a_1z+a_0$ is a real polynomial that is Hurwitz stable with $a_n>0$, then $a_{j}>0$ for all $j=0,\ldots, n$.
\end{lemma}
\begin{proof}
    As $p$ is a Hurwitz stable polynomial, then $p(\lambda)=0$ implies $\operatorname{Re}(\lambda)<0$. Since the nonreal zeros of the real polynomial $p$ come in complex conjugate pairs, we can write
    \begin{gather*}
        p(z)=a_n\prod_{\lambda\in \mathbb{R}:p(\lambda)=0, \lambda<0}(z-\lambda)\prod_{\lambda\in \mathbb{C}:p(\lambda)=0,\operatorname{Re}\lambda<0}(z-\lambda)(z-\overline{\lambda})\\
        =a_n\prod_{\lambda\in \mathbb{R}:p(\lambda)=0, \lambda<0}[z+(-\lambda)]\prod_{\lambda\in \mathbb{C}:p(\lambda)=0,\operatorname{Re}\lambda<0}[z^2+(-2\operatorname{Re}\lambda) z+|\lambda|^2].
    \end{gather*}
    The result now follows immediately from this. 
\end{proof}

Next, observe that the Kharitonov polynomials $k_j,j=1,2,3,4$ in (\ref{DefKharPolys}) can be written as
\begin{gather}
    \label{DefRealImagPartKharPolysPartI}\left\{\begin{matrix}
        \kappa_1(\omega)=k_1(i\omega)=h_{-}(\omega)+ig_{-}(\omega),\\
        \kappa_2(\omega)=k_2(i\omega)=h_{+}(\omega)+ig_{-}(\omega),\\
        \kappa_3(\omega)=k_3(i\omega)=h_{+}(\omega)+ig_{+}(\omega),\\
        \kappa_4(\omega)=k_4(i\omega)=h_{-}(\omega)+ig_{+}(\omega),
    \end{matrix}\right.
\end{gather}
where $h_{\pm}(\omega)$ and $g_{\pm}(\omega)$ are the real polynomials given by 
\begin{gather}
    \label{DefRealImagPartKharPolysPartII}\left\{\begin{matrix}
        g_{-}(\omega)=\frac{\kappa_1(\omega)-\overline{\kappa_1(\overline{\omega})}}{2i}=\frac{\kappa_2(\omega)-\overline{\kappa_2(\overline{\omega})}}{2i}=a_1^-\omega-a_3^+\omega^3+a_5^-\omega^5-\cdots,\vspace{0.5em}\\
        \!\!\!\!h_{+}(\omega)=\frac{\kappa_2(\omega)+\overline{\kappa_2(\overline{\omega})}}{2}=\frac{\kappa_3(\omega)+\overline{\kappa_3(\overline{\omega})}}{2}=a_0^+-a_2^-\omega^2+a_4^+\omega^4-\cdots,\vspace{0.5em}\\
        g_{+}(\omega)=\frac{\kappa_3(\omega)-\overline{\kappa_3(\overline{\omega})}}{2i}=\frac{\kappa_4(\omega)-\overline{\kappa_4(\overline{\omega})}}{2i}=a_1^+\omega-a_3^-\omega^3+a_5^+\omega^5-\cdots,\vspace{0.5em}\\
        \!\!\!\!h_{-}(\omega)=\frac{\kappa_1(\omega)+\overline{\kappa_1(\overline{\omega})}}{2}=\frac{\kappa_4(\omega)+\overline{\kappa_4(\overline{\omega})}}{2}=a_0^--a_2^+\omega^2+a_4^-\omega^4-\cdots.
    \end{matrix}\right.
\end{gather}

From this it is easy to prove the following two corollaries (cf.\ \cite{00KT}, \cite{88SD},  \cite[Footnote 1]{89MAD}, \cite{96DH}, and \cite[Lemma 4.1.31]{05HP}), the latter of which gives a characterization of the geometry 
associated with the value set of $P$ evaluated at $z=i\omega$ for fixed $\omega\geq 0$, namely, that the value set
\begin{gather}
    K(i\omega)=\{p(i\omega):p\in P\}\label{DefKharRectangle}
\end{gather} 
has the geometry of a rectangle (as first pointed out explicitly by S.\ Dasgupta \cite{88SD} together with \cite[footnote 1]{89MAD}), called the \textit{Kharitonov rectangle} in \cite{94RB} (cf.\ Fig.\ \ref{FigKharitonovRectangleUndirected}).
\begin{corollary}\label{CorKharPolysPositiveCoeffs}
     If the Kharitonov polynomials $k_j,j=1,2,3,4$ in (\ref{DefKharPolys}) are all Hurwitz stable and $a_n^+>0$, then either $n=0$ and $a_0^->0$, or $n\geq 1$ and
     \begin{gather*}
         a_n^-\geq0\text{ and } a_j^->0,\text{ for all }i=0,\ldots, n-1.
     \end{gather*}
\end{corollary}
\begin{proof}
The proof follows immediately from Lemma \ref{LemStodolasRule} by considering (\ref{DefRealImagPartKharPolysPartI}) and (\ref{DefRealImagPartKharPolysPartII}).
\end{proof}

\begin{figure}[!ht]
   \centering
   \begin{tikzpicture}[scale=1]
        \draw[>=stealth, ->] (0,-0.5) --(0,4.5) node[above] {$\operatorname{Im}$};
        \draw[>=stealth, ->] (-0.5,0) --(6.5,0) node[right] {$\operatorname{Re}$};
        \draw[thick] (1.5,1.5) -- (1.5,3.5); 
        \draw[thick] (5.5,1.5) -- (5.5,3.5); 
        \draw[thick] (1.5,1.5) -- (5.5,1.5); 
        \draw[thick] (1.5,3.5) -- (5.5,3.5); 
        \draw[thick, dashed] (5.5,0) -- (5.5,1.5); 
        \draw[dashed, thick] (0,3.5) -- (1.5,3.5); 
        \draw[dashed, thick] (1.5,0) -- (1.5,1.5); 
        \draw[dashed, thick] (0,1.5) -- (1.5,1.5); 
        \draw [color=black] (0,3.5) node[left] {$g_+(\omega)$};
        \draw (1.5,0) node[below] {$h_-(\omega)$};
        \draw (0,1.5)  node[left] {$g_-(\omega)$};
        \draw [color=black] (5.5,0) node[below] {$h_+(\omega)$};
        \draw [color=black, fill=black] (5.5,3.5)circle(0.05) node[right, yshift=0.26cm] {$k_3(i\omega)$};
        \draw [color=black, fill=black] (1.5,3.5)circle(0.05) node[left, yshift=0.26cm, xshift=0.125cm] {$k_4(i\omega)$};
        \draw [color=black, fill=black] (5.5,1.5)circle(0.05) node[right, yshift=-0.26cm] {$k_2(i\omega)$};
        \draw [color=black, fill=black] (1.5,1.5)circle(0.05) node[left, yshift=-0.26cm, xshift=0.125cm] {$k_1(i\omega)$};
        \draw [] (3.5,2.5) node[] {$K(i\omega)$};
        \draw[color=gray] (1.5,2.75) -- (2.35,3.5); 
        \draw[color=gray] (1.5,2) -- (3.1,3.5); 
        \draw[color=gray] (1.75,1.5) -- (2.55,2.25); 
        \draw[color=gray] (2.5,1.5) -- (3.25,2.15); 
         \draw[color=gray] (3.25,1.5) -- (4,2.15); 
         \draw[color=gray] (4,1.5) -- (5.5,2.75); 
        \draw[color=gray] (4.75,1.5) -- (5.5,2.1); 
        \draw[color=gray] (3.45,3) -- (4,3.5); 
        \draw[color=gray] (4.2,3) -- (4.75,3.5); 
        \draw[color=gray] (4.6,2.7) -- (5.5,3.5); 
    \end{tikzpicture}  
    \caption{
    Under the assumptions in Corollary \ref{CorKharRectUnDirected}, this figure depicts the Kharitonov rectangle $K(i\omega)=\{p(i\omega):p\in P\}=\{x+iy:h_-(\omega)\leq x\leq h_+(\omega), g_-(\omega)\leq y\leq g_+(\omega)\}$ 
     in the complex plane with vertices the Kharitonov polynomials $k_j(i\omega)$, $j=1,2,3,4$, evaluated at $z=i\omega$ for a fixed $\omega\geq 0$.
    }
    \label{FigKharitonovRectangleUndirected}
\end{figure}
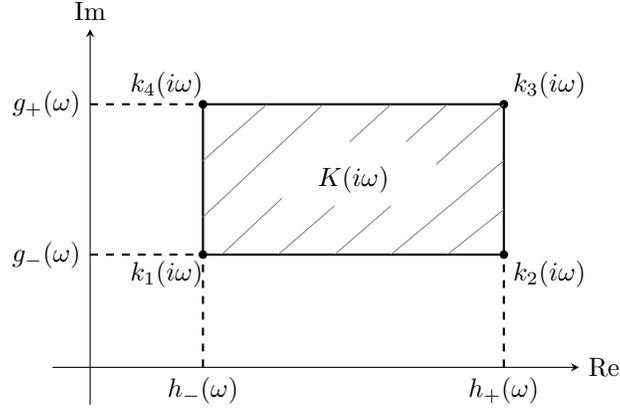

\begin{corollary}\label{CorKharRectUnDirected}
     If the Kharitonov polynomials $k_j,j=1,2,3,4$ in (\ref{DefKharPolys}) are all Hurwitz stable and $a_n^+>0$ then, for every $\omega\geq 0$,
     \begin{gather}
         K(i\omega)=\{x+iy:h_-(\omega)\leq x\leq h_+(\omega), g_-(\omega)\leq y\leq g_+(\omega)\},
     \end{gather}
     where $K(i\omega)$ is the Kharitonov rectangle defined by \eqref{DefKharRectangle}.
\end{corollary}
\begin{proof}
For a fixed $\omega\geq 0$, we define the rectangle
\begin{gather*}
    R(\omega)=\{x+iy:h_-(\omega)\leq x\leq h_+(\omega), g_-(\omega)\leq y\leq g_+(\omega)\}.
\end{gather*} 
Then from Corollary \ref{CorKharPolysPositiveCoeffs} we have 
\begin{align*}
    0\leq a_j^-\leq a_j\leq a_j^+,\;\; j=0,\ldots, n,
\end{align*} 
and hence it follows by the identities (\ref{DefRealImagPartKharPolysPartI}) and (\ref{DefRealImagPartKharPolysPartII}) that
\begin{align*}
    h_-(\omega)=a_0^--a_2^+\omega^2+a_4^-\omega^4-\cdots&\leq a_0-a_2\omega^2+a_4\omega^4-\cdots \\
    &\leq  a_0^+-a_2^-\omega^2+a_4^+\omega^4-\cdots=h_+(\omega),\\
    g_-(\omega)=a_1^-\omega-a_3^+\omega^3+a_5^-\omega^5-\cdots&\leq a_1\omega-a_3\omega^3+a_5\omega^5-\cdots \\
    &\leq  a_1^+\omega-a_3^-\omega^3+a_5^+\omega^5-\cdots=g_+(\omega),
\end{align*}
which implies the inclusion $K(i\omega)\subseteq R(\omega)$\footnote{We only need the inclusion $K(i\omega)\subseteq R(\omega)$ in the proof of Kharitonov's Theorem. 
}. Next, by (\ref{DefRealImagPartKharPolysPartI}), it follows that the vertices $\kappa_j(
\omega
), j=1,2,3,4$, of the rectangle $R(\omega)$ are in $K(i\omega)$ and as $K(i\omega)$ is contained in $R(\omega)$,  we need only prove that $K(i\omega)$ is a convex set. But this follows immediately from the definition of the interval polynomial $P$ in (\ref{DefIntervalPolyOrdern}) and the fact that the Cartesian product of intervals $[a_0^-,a_0^+]\times\cdots\times [a_n^-,a_n^+]$ is a convex set.
\end{proof}

Next, define the \textit{Wronskian} of two polynomials $h(\omega)$ and $g(\omega)$, denoted by  $W[h,g](\omega)$, to be the polynomial
    \begin{gather}
        W[h,g](\omega)=h(\omega)g'(\omega)-h'(\omega)g(\omega).\label{DefWronkianOfhandg}
    \end{gather} 
The following lemma is well-known (see, for instance, \cite[Proposition 3.4.5]{05HP}).
\begin{lemma}\label{LemPosWronskian}
    Let $f(\omega)=h(\omega)+ig(\omega)$ be a non-constant polynomial, where $h(\omega)$ and $g(\omega)$ are real polynomials. If all the roots of $f$ are in the open upper-half plane then, for every $\omega \in \mathbb{R}$,
    \begin{gather}
        W[h,g](\omega)=h(\omega)g'(\omega)-h'(\omega)g(\omega)>0.\label{SignOfWronkianOfhandg}
    \end{gather}
\end{lemma}
\begin{proof}
This follows immediately from the hypotheses by considering the partial fraction decomposition of $f'/f$:
    \begin{gather*}
        \frac{f'(\omega)}{f(\omega)}=\sum_{\lambda\in \mathbb{C}:f(\lambda)=0}\frac{1}{\omega-\lambda}=\sum_{\lambda\in \mathbb{C}:\operatorname{Im}\lambda>0,f(\lambda)=0}\frac{1}{\omega-\lambda}, \text{ if }f(\omega)\not=0;\\
        \operatorname{Im}\left[\frac{f'(\omega)}{f(\omega)}\right]= \sum_{\lambda\in \mathbb{C}:\operatorname{Im}\lambda>0,f(\lambda)=0}\frac{\operatorname{Im}\lambda-\operatorname{Im}\omega}{|\omega-\lambda|^2}>0,\; \text{if }\operatorname{Im}\omega\leq 0;\\
        W[h,g](\omega)=\operatorname{Im}[f'(\omega)\overline{f(\omega)}]=|f(\omega)|^2\operatorname{Im}\left[\frac{f'(\omega)}{f(\omega)}\right]>0,\text{ if }\omega\in \mathbb{R}. \qedhere
    \end{gather*}
\end{proof}

\section{Proof of Kharitonov's Theorem}\label{sec:PrfKharThm}
We are now ready to prove Kharitonov's Theorem. 
\begin{proof}[Proof of Theorem \ref{ThmKhar}]
($\Rightarrow$): If the interval polynomial $P$ in \eqref{DefIntervalPolyOrdern} is Hurwitz stable, then so are the Kharitonov polynomials $k_j, j=1,2,3,4$, in \eqref{DefKharPolys} since $k_j\in P$ for $j=1,2,3,4$.

($\Leftarrow$): Conversely, suppose the Kharitonov polynomials in \eqref{DefKharPolys} are all Hurwitz stable. We may assume that $a_n^+>0$, for otherwise we can consider the interval polynomial $-P$ instead. 
We give a proof by contradiction. Assume that there exists $p\in P$ which is not Hurwitz stable. As $a_n^+\not=0$, there exists a Kharitonov polynomial 
 in (\ref{DefKharPolys}) with degree $n$, which we denote by $k(z)$. Consider the family of polynomials given by the convex combination 
\begin{gather*}
    p_t(z)=(1-t)k(z)+tp(z), \hspace{0.2in}t\in[0,1],
\end{gather*}
of the polynomials $k(z)$ and $p(z)$.
 Then by Corollary \ref{CorKharPolysPositiveCoeffs} and since $k(z)$ is a degree $n$ Hurwitz stable polynomial, it follows that the hypotheses of Lemma \ref{LemBddCrossnThmAndConvexCombStablePolysEndPts} hold for this family. Furthermore, since $p_1(z)=p(z)$ is not Hurwitz stable, Lemma \ref{LemBddCrossnThmAndConvexCombStablePolysEndPts} implies there exists a $(t_*,\omega_*)\in (0,1]\times\mathbb{R}$ such that $p_{t_*}(i\omega_*)=0$. By Corollary \ref{CorKharRectUnDirected}, there exists $\alpha_1,\alpha_2,\alpha_3,\alpha_4\in [0,1]$ with $\sum_{j=1}^4\alpha_j=1$ such that
\begin{gather*}
   p_{t_*}(i\omega_*)=\sum_{j=1}^4\alpha_jk_j(i\omega_*).
\end{gather*}
In terms of the real polynomials $h_{\pm}(\omega), g_{\pm}(\omega)$ from (\ref{DefRealImagPartKharPolysPartI}) and (\ref{DefRealImagPartKharPolysPartII}), this implies that
\begin{gather*}
\sum_{j=1}^4\alpha_jk_j(i\omega)=h(\omega)+ig(\omega),
\end{gather*}
where $h(\omega)$ and $g(\omega)$ are the real polynomials given by
\begin{gather*}
    h(\omega)=(\alpha_1+\alpha_4)h_-(\omega)+(\alpha_2+\alpha_3)h_+(\omega),\\
     g(\omega)=(\alpha_1+\alpha_2)g_-(\omega)+(\alpha_3+\alpha_4)g_+(\omega).
\end{gather*}
It follows immediately from this and $h(\omega_*)=g(\omega_*)=0$ that
$$0=W[h,g](\omega_*)=(\alpha_1+\alpha_4)(\alpha_1+\alpha_2)W[h_-,g_-](\omega_*) \leqno 
\refstepcounter{equation}(\the\value{equation})\label{WronksianIsZero}$$
\vspace{-1.9em}$$\hspace{7em}+\;(\alpha_1+\alpha_4)(\alpha_3+\alpha_4)W[h_-,g_+](\omega_*)$$
\vspace{-1.5em}$$\hspace{7em}+\;(\alpha_2+\alpha_3)(\alpha_1+\alpha_2)W[h_+,g_-](\omega_*)$$
\vspace{-1.5em}$$\hspace{7em}+\;(\alpha_2+\alpha_3)(\alpha_3+\alpha_4)W[h_+,g_+](\omega_*).$$
It follows by Lemma \ref{LemPosWronskian} that if all the Kharitonov polynomials in (\ref{DefKharPolys}) are non-constant, then each of $W[h_-,g_-](\omega_*),W[h_-,g_+](\omega_*),W[h_+,g_-](\omega_*)$, and $W[h_+,g_+](\omega_*)$ are positive so \eqref{WronksianIsZero} would imply that $\alpha_j=0$ for $j=1,2,3,4$, a contradiction that $\sum_{j=1}^4\alpha_j=1$. Hence at least one of the Kharitonov polynomials in (\ref{DefKharPolys}) is constant, but as they cannot all be constants (since $p$ is not Hurwitz stable), this implies by Corollary \ref{CorKharPolysPositiveCoeffs} that $n=1$ and $a_n^-=0$. In this case, a similar argument as above implies $\alpha_j=0$ for $j=1,2,3,4$, a contradiction. This proves the theorem.\qedhere
\end{proof}

\section*{Declarations}

\subsection*{Funding} The work of AW was supported by the Simons Foundation under Grant MPS-TSM-00002799 and by the National Science Foundation under Grant DMS-2410678.
\subsection*{Conflicts of interest/Competing interests} Not applicable.
\subsection*{Availability of data and material} Not applicable.
\subsection*{Code availability} Not applicable.

\bibliography{PrfKharitonovsThmBib}

\end{document}